\documentclass[12pt,a4paper,oneside]{amsart}
\usepackage[a4paper]{geometry}
\usepackage{mathptmx} 
\DeclareMathAlphabet{\mathcal}{OMS}{cmsy}{m}{n} 

\usepackage{amssymb,mathrsfs,perpage}

\MakePerPage[2]{footnote}

\DeclareMathOperator{\id}{id}
\DeclareMathOperator{\im}{Im}

\newtheorem{lemma}{Lemma}
\newtheorem*{proposition}{Proposition}
\newtheorem*{corollary}{Corollary}
\newtheorem{question}{Question}

\begin{document}

\title{Non-semigroup gradings of associative algebras}
\author{Pasha Zusmanovich}
\address{
Department of Mathematics, University of Ostrava, Ostrava, Czech Republic
}
\email{pasha.zusmanovich@osu.cz}
\date{last revised January 3, 2017}
\thanks{arXiv:1609.03924}

\begin{abstract}
It is known that there are Lie algebras with non-semigroup gradings, i.e. such 
that the binary operation on the grading set is not associative. We provide a 
similar example in the class of associative algebras.
\end{abstract}

\maketitle

\section*{Introduction}

Let $A$ be a (generally, not necessarily associative) algebra, and 
$A = \bigoplus_{g \in \Gamma} A_g$ its grading over a set $(\Gamma, *)$, i.e. 
$*: \Gamma \times \Gamma \to \Gamma$ is a partial binary operation defined for 
each pair $(g,h)$ such that $A_g A_h \ne 0$, in which case  
$A_g A_h \subseteq A_{g * h}$. How the identities satisfied by the algebra $A$
are related to identities satisfied by the grading set $\Gamma$? Since the
operation $*$ on $\Gamma$ is partial, in the latter case it makes sense to speak
about the (im)possibility to complete $*$ in such a way that it will satisfy 
that or another identity, or, in more strict terms, about the (im)possibility of
embedding of $(\Gamma, *)$ into an appropriate magma\footnote{
``Magma'' means a set with an (everywhere defined) binary operation on it, 
without any additional conditions. In the older literature, the term 
``groupoid'' was used instead, but since then the latter term was taken by 
category theorists.}
$(G, \cdot)$ such that $g * h = g \cdot h$ whenever $A_g A_h \ne 0$.

It is immediate that commutativity or anticommutativity of $A$ implies that 
$\Gamma$ can be embedded into a commutative magma. Elementary manipulations 
involving homogeneous components $A_g$'s of graded Lie and associative algebras
may suggest that both Jacobi identity and associativity of the algebra $A$ are 
strongly connected with the associativity of the grading set $\Gamma$. In the 
Lie case, it was believed for a while (and even claimed in an influential paper
\cite{zass} as Theorem 1(a)) that each grading of a Lie algebra is a 
\emph{semigroup grading}, i.e. the grading set $(\Gamma, *)$ can be embedded 
into a semigroup. This is indeed so for all gradings of Lie and associative 
algebras appearing naturally (root space decompositions with respect to a Cartan
subalgebra, gradings arising from various group or Hopf algebra actions on the 
algebra, $\mathbb Z$-gradings providing connection between Lie and Jordan algebras, semigroup algebras and their twisted variants, grading by Pauli 
matrices motivated by physics, etc.). However, in \cite{elduque-1} and 
\cite{elduque-2} examples of non-semigroup gradings of Lie algebras were given. 
The aim of this note is to provide an example of a non-semigroup grading of an
associative algebra. This is done in \S \ref{sec-nonsemigr} following approach 
of \cite[\S 3]{delta}, where it was shown how non-semigroup gradings of Lie 
algebras can be constructed using $\delta$-derivations. \S \ref{sec-quest} contains some further questions.

\section{An example of a non-semigroup grading}\label{sec-nonsemigr}

In the associative case, instead of $\delta$-derivations we may consider a 
slightly more general notion of $(\delta,\gamma)$-derivations, i.e. linear maps
$D: A \to A$ on an algebra $A$ such that
$$
D(xy) = \delta D(x)y + \gamma x D(y)
$$
for any $x,y \in A$, and some fixed elements of the ground field 
$\delta, \gamma$. In the Lie case, due to anticommutativity, any such condition
implies that either $\delta = \gamma$, i.e. $D$ is a 
$\delta$-derivation, or that $D$ is an element of ``generalized centroid'', i.e.
$$
D(xy) = (\delta + \gamma)D(x)y = (\delta + \gamma)xD(y)
$$
for any $x,y\in A$, the latter condition being too restrictive to be 
interesting. (The same dichotomy holds for commutative algebras).

\begin{lemma}\label{lemma-a}
Let $A$ be a finite-dimensional algebra over an algebraically closed field $K$,
and $A = \bigoplus_{\lambda \in K} A_\lambda$ is the root space 
decomposition with respect to an $(\delta,\gamma)$-derivation of $A$. Then 
$A_\lambda A_\mu \subseteq A_{\delta\lambda + \gamma\mu}$ for any 
$\lambda, \mu \in K$.
\end{lemma}

Note that the algebra $A$ here and below is not assumed to be associative, or 
Lie, or to satisfy any other distinguished identity.

\begin{proof}
It is trivial to check that if $x$ and $y$ are eigenvectors of an 
$(\delta,\gamma)$-derivation of $A$, corresponding to eigenvalues $\lambda$ and
$\mu$ respectively, then the product $xy$ is an eigenvector corresponding to 
$\delta\lambda + \gamma\mu$ (or zero, if $\delta\lambda + \gamma\mu$ is not an 
eigenvalue). Then proceed by induction on the sum of multiplicities of the 
respective eigenvalues, exactly the same way as in, for example, 
\cite[Chapter III, \S 2]{jacobson}.
\end{proof}

The following is a slightly modified ``nonassociative'' analogue of the 
Lie-algebraic statement \cite[Proposition 3.1]{delta}.

\begin{proposition}
Let $A$ be a finite-dimensional algebra over an algebraically closed field, and 
$D$ an $(\delta,\gamma)$-derivation of $A$. Suppose that there are roots 
$\lambda, \mu, \eta, \theta, \xi$ (not necessarily distinct) in the root space 
decomposition of $A$ with respect to $D$ such that 
\begin{gather}
0 \ne A_\lambda A_\eta \subseteq A_\theta, \quad A_\theta A_\mu \ne 0 ,
\label{eq-cond}
\\
0 \ne A_\eta A_\mu \subseteq A_\xi, \quad A_\lambda A_\xi \ne 0 ,
\label{eq-cond1}
\end{gather}
and $(\delta^2 - \delta)\lambda \ne (\gamma^2 - \gamma)\mu$. Then the said root
space decomposition is a non-semigroup grading of $A$.
\end{proposition}

Note that the conditions (\ref{eq-cond}) and (\ref{eq-cond1}) are somewhat 
weaker than $(A_\lambda A_\eta) A_\mu \ne 0$ and 
$A_\lambda (A_\eta A_\mu) \ne 0$, respectively.

\begin{proof}
The conditions (\ref{eq-cond}) and (\ref{eq-cond1}) ensure that both expressions
$(\lambda * \eta) * \mu$ and $\lambda * (\eta * \mu)$ are defined. If the root 
space decomposition of $A$ with respect to $D$ is a semigroup grading, then 
these two expressions are equal: 
$(\lambda * \eta) * \mu = \lambda * (\eta * \mu)$. By Lemma \ref{lemma-a}, this 
equality is equivalent to $(\delta^2 - \delta)\lambda = (\gamma^2 - \gamma)\mu$,
a contradiction.
\end{proof}

\begin{corollary}
The conclusion of Proposition holds in each of the following cases:
\begin{enumerate}
\item $\delta = \gamma \ne 0,1$, and $\lambda \ne \mu$;
\item $\delta \ne \gamma$, $\delta + \gamma \ne 1$, and $\lambda = \mu \ne 0$.
\end{enumerate}
\end{corollary}

\begin{proof}
Obvious.
\end{proof}

Now we will provide an example of a family of associative algebras having 
$\delta$-derivations as in heading (i) of the Corollary, and hence admitting a 
non-semigroup grading. Let $V$ be a vector space over a field $K$, and 
$f_L,f_R,g_L,g_R: V \to V$ be four linear maps. Consider the vector space direct
sum 
\begin{equation*}
Ke \oplus Ka \oplus V \oplus V^\prime ,
\end{equation*}
where $Ke$ and $Ka$ are one-dimensional vector spaces spanned by elements $e$
and $a$ respectively, and $V^\prime$ is a second copy of $V$, identified with
$V$ via a nondegenerate linear map $v \mapsto v^\prime$. Define the 
multiplication on this direct sum as follows:
\begin{gather*}
e^2 = e , \quad 
a^2 = 0 , \quad 
av = f_L(v)^\prime , \quad va = f_R(v)^\prime , \quad  
av^\prime = g_L(v) , \quad v^\prime a = g_R(v) ,
\end{gather*}
where $v \in V$, and the rest of the products between basic elements are zero. 
The associativity of the so defined algebra, let us denote it as 
$A(f_L,f_R,g_L,g_R)$, is equivalent to the following conditions:
\begin{align*}
f_L \circ g_L &= g_L \circ f_L = 0 \\
f_R \circ g_R &= g_R \circ f_R = 0 \\
g_R \circ f_L &= g_L \circ f_R     \\
f_R \circ g_L &= f_L \circ g_R .
\end{align*}

\begin{lemma}\label{lemma-der}
Suppose that each of the maps $f_L,f_R,g_L,g_R$ is nonzero, and 
$(\delta,\gamma) \ne (0,0)$. Then each $(\delta,\gamma)$-derivation $D$ of the 
algebra $A(f_L,f_R,g_L,g_R)$ is of the following form:
\begin{align*}
D(e)\phantom{\prime} &= \begin{cases}
0       & \text{ if } \delta + \gamma \ne 1 \\
\beta e & \text{ if } \delta + \gamma = 1
\end{cases} 
\\
D(a)\phantom{\prime} &= \alpha a + v_a + w_a^\prime \\
D(v)\phantom{\prime} &= \varphi(v) + \psi(v)^\prime , \quad v \in V         \\
D(v^\prime) &= \widetilde{\varphi}(v) + \widetilde{\psi}(v)^\prime 
\end{align*}
where $\alpha, \beta \in K$, $v_a, w_a \in V$, 
$\varphi,\widetilde{\varphi},\psi,\widetilde{\psi}: V \to V$ are linear maps,
and the following conditions are satisfied: 
\begin{alignat*}{2}
&(\delta f_R + \gamma f_L)(v_a) &\>= 0 \\
&(\delta g_R + \gamma g_L)(w_a) &\>= 0
\end{alignat*}
and
\begin{align*}
\widetilde\varphi \circ f_L &= \gamma g_L \circ \psi                   \\
\widetilde\psi \circ f_L &= \delta\alpha f_L + \gamma f_L \circ \varphi \\
\widetilde\varphi \circ f_R &= \delta g_R \circ \psi \\
\widetilde\psi \circ f_R &= \gamma\alpha f_R + \delta f_R \circ \varphi \\
\varphi \circ g_L &= \delta\alpha g_L + \gamma g_L \circ \widetilde\psi \\
\psi \circ g_L &= \gamma f_L \circ \widetilde\varphi \\
\varphi \circ g_R &= \gamma\alpha g_R + \delta g_R \circ \widetilde\psi \\
\psi \circ g_R &= \delta f_R \circ \widetilde\varphi .
\end{align*}
\end{lemma}

\begin{proof}
Direct calculations.
\end{proof}

The non-vanishing conditions of Lemma \ref{lemma-der} are merely technical ones,
to avoid consideration of numerous degenerate tedious cases.

We may specialize this setup in many different ways to get an example of an 
algebra having a $(\delta,\gamma)$-derivation satisfying the condition of 
Proposition or its Corollary, and hence admitting a non-semigroup 
grading. One of the easiest ways is to set $f_L = f_R = g_L = g_R = f$, where
$f \circ f = 0$ (say, $V$ is $2$-dimensional, and $f$ has the matrix 
$\begin{pmatrix} 0 & 1 \\ 0 & 0 \end{pmatrix}$ in the canonical basis), 
$\delta = \gamma = -1$, $\alpha = \beta = 0$, $v_a = w_a = 0$, and
$\psi = \widetilde\varphi = 0$, $\varphi = \id_V$, $\widetilde\psi = -\id_V$.
Then $D$ from Lemma \ref{lemma-der} is a $(-1)$-derivation (or, 
\emph{antiderivation}) of the algebra $A(f,f,f,f)$. The eigenvalues of $D$ are 
$0, 1, -1$, with eigenspaces $A_{0} = Ke \oplus Ka$, $A_1 = V$, and 
$A_{-1} = V^\prime$. Then by heading (i) of Corollary, the root space 
decomposition $A(f,f,f,f) = A_0 \oplus A_1 \oplus A_{-1}$ is a non-semigroup 
grading. This fact can be also verified directly: as 
$A_0^2 = Ke \subset A_0$, $A_0 A_1 = A_1 A_0 = (\im f)^\prime \subset A_{-1}$, and 
$A_0 A_{-1} = A_{-1} A_0 = \im f \subset A_1$, we have the following (partial)
operation on the grading set:
$$
0*0 = 0 , \quad 0 * 1 = 1 * 0 = -1, \quad 0 * (-1) = (-1) * 0 = 1 ,
$$
what contradicts associativity: 
$$
1 = 0*(-1) = (0*0)*(-1) \ne 0*(0*(-1)) = 0*1 = -1 .
$$
(Note that this is the same non-associative grading set as in the Lie-algebraic
example in \cite{elduque-2}).

The algebra $A(f,f,f,f)$ is, obviously, commutative, with a commutative grading.
By modifying this example to make the maps $f_L$, $f_R$, $g_L$, $g_R$ different,
it is possible to get various examples of associative non-commutative algebras 
with a non-semigroup grading, commutative or not. The relevant calculations are
trivial, but somewhat cumbersome, and are left to the interested reader.

\section{Further questions}\label{sec-quest}

If $L = \bigoplus_{g \in \Gamma}  L_g$ is a Lie algebra graded by an 
\emph{abelian group} $\Gamma$, then its universal enveloping algebra $U(L)$ is a
$\Gamma$-graded associative algebra, with the graded components $U(L)_g$ 
linearly spanned by monomials of the form $x_1 \dots x_k$, where 
$x_i \in L_{g_i}$ and $\sum_{i=1}^k g_i = g$ (see, e.g., 
\cite[Theorem 4.3]{strade-farn}). 

The algebra $U(L)$ is infinite-dimensional, what, perhaps, is not that 
interesting in our context. In the positive characteristic it is possible, 
however, to define the same grading on the finite-dimensional \emph{restricted}
universal enveloping algebra of a graded restricted Lie algebra. However, the 
facts that multiplication in $\Gamma$ is defined everywhere, and is associative,
are crucial in this construction, and it is unclear how to extend or modify it 
to grading by an arbitrary set $\Gamma$.

\begin{question}
Is it possible to construct a grading of the (restricted) universal enveloping
algebra, given (arbitrary, not necessarily semigroup) grading of the underlying 
Lie algebra?
\end{question}

A positive answer to this question will produce a plethora of non-semigroup 
gradings of finite-dimensional associative algebras in positive characteristic,
different from those exhibited in \S \ref{sec-nonsemigr}: take any of the 
examples from \cite{elduque-1} or \cite{elduque-2} over a field of 
characteristic $p>0$, pass, if necessary, to the $p$-envelope, and consider the
restricted universal enveloping algebra.

\begin{question}
What is the minimal dimension of an associative algebra admitting a 
non-semigroup grading?
\end{question}

It is, probably, possible to prove, following the approach of 
\cite[Theorem in \S 1]{elduque-2}, and classification of low-dimensional 
associative algebras, that any grading of an associative algebra of dimension 
$\le 3$ is a semigroup grading. Since the underlying algebra is not necessarily
commutative, there are apriori much more possibilities for a noncommutative
partial operation on a $2$- and $3$-element grading set. The relevant 
calculations should be straightforward, but definitely cumbersome.

We also failed to find examples of non-semigroup gradings of associative 
algebras of dimension $4$ and $5$. The minimal dimension of an algebra with 
non-semigroup grading following the scheme of \S \ref{sec-nonsemigr} is $6$.

By analogy with the question about gradings of simple Lie algebras from 
\cite{elduque-1}, one may ask

\begin{question}
Is it true that any grading of a full matrix algebra is a semigroup grading?
\end{question}

Note that this question cannot be approached by constructing an appropriate 
$(\delta,\gamma)$-derivation as in \S \ref{sec-nonsemigr}: it is easy to see 
that any $(\delta,\gamma)$-derivation of a full matrix algebra is either an
(inner) derivation, or a scalar multiple of the identity map (see, e.g., 
\cite[Theorem 1]{shest} for a slightly more general statement).

Finally, note that, in principle, the same approach as in 
\S \ref{sec-nonsemigr} may be used to construct examples of non-semigroup 
gradings in varieties of algebras satisfying other identities of degree $3$ 
(like Leibniz, Zinbiel, left-symmetric, Lie-admissible, Alia algebras, etc.). 
Another interesting topic would be to explore the question from the point of view of operadic Koszul duality: 
for example, does the presence/absence of non-semigroup gradings of algebras
over a binary quadratic operad $\mathscr P$ entails the same for algebras over 
the operad Koszul dual to $\mathscr P$?

\section*{Acknowledgements}

Thanks are due to Miroslav Korbel\'a\v{r} for asking questions which prompted 
me to write this note. This work was supported by 
the Statutory City of Ostrava (grant 0924/2016/Sa\v{S}), and
the Ministry of Education and Science of the Republic of Kazakhstan 
(grant 0828/GF4).

\end{document}